\documentclass{amsart}

\title{On the ampleness of positive CR line bundles over Levi-flat manifolds}
\author{Masanori Adachi}
\address{Graduate~School~of~Mathematics, Nagoya~University, Furo-cho Chikusa-ku Nagoya 464-8602, Japan}
\email{m08002z@math.nagoya-u.ac.jp}
\date{October 17, 2013}
\thanks{This is the author's final version of an article to appear in Publ. Res. Inst. Math. Sci.}

\usepackage{amsmath,amsthm,amssymb,amsrefs,latexsym}

\newtheorem{Thm}{Theorem}[section]
\newtheorem{Cor}[Thm]{Corollary}
\newtheorem{Lem}[Thm]{Lemma}
\newtheorem{Prop}[Thm]{Proposition}
\newtheorem*{Thm*}{Theorem}

\newtheorem*{MainTheorem}{Main Theorem}

\theoremstyle{definition}
\newtheorem{Def}[Thm]{Definition}
\newtheorem{Rem}[Thm]{Remark} 
\newtheorem{Ex}[Thm]{Example}
\newtheorem{Que}{Question}

\numberwithin{equation}{section}  

\def\N{\mathbb{N}}
\def\R{\mathbb{R}}
\def\C{\mathbb{C}}
\def\D{\mathbb{D}}
\def\Cont{\mathcal{C}}
\def\CP{\mathbb{CP}}
\def\Aut{\mathrm{Aut}}
\def\End{\mathrm{End}}
\def\Ker{\mathrm{Ker}}
\def\id{\mathrm{Id}}
\def\rank{\mathrm{rank}}
\newcommand{\eqdef}{\mathrel{\mathop:}=}
\newcommand{\timesrho}{\times_\rho}

\def\d'{\partial}
\def\dbar{\overline{\partial}}
\def\bd{\partial}

\allowdisplaybreaks 

\makeatletter
\newsavebox{\@brx}
\newcommand{\llangle}[1][]{\savebox{\@brx}{\(\m@th{#1\langle}\)}%
  \mathopen{\copy\@brx\kern-0.5\wd\@brx\usebox{\@brx}}}
\newcommand{\rrangle}[1][]{\savebox{\@brx}{\(\m@th{#1\rangle}\)}%
  \mathclose{\copy\@brx\kern-0.5\wd\@brx\usebox{\@brx}}}
\makeatother

\begin{document}

\maketitle

\begin{abstract}
We give an example of a compact Levi-flat CR 3-manifold with a positive-along-leaves 
CR line bundle which is not ample with respect to transversely infinitely differentiable 
CR sections.  
This example shows that we cannot improve the regularity of Kodaira type embedding theorem 
for compact Levi-flat CR manifolds obtained by Ohsawa and Sibony.
\end{abstract}


\section{Introduction}
We are going to study function theory on Levi-flat CR 3-manifolds, 
i.e., 3-manifolds foliated by Riemann surfaces. 
If we look them as families of Riemann surfaces, we can expect analogy with 
classical theory on Riemann surfaces, such as the Riemann-Roch theorem.  
But here is a new ingredient: dynamics of the foliation. 
We should face subtle interaction between complexity of Levi foliations 
and existence of CR functions, and especially in the case where the 
Levi-flat CR 3-manifold is realized as a real hypersurface in a complex surface, 
it should be reflected on pseudoconvexity of the complement, and complex geometry of the ambient space.
There are several attempts toward this direction. We refer the reader to the works of 
Inaba \cite{inaba1992} and  Barrett \cite{barrett1992}.

We investigate such phenomenon in a problem on an analogue of Kodaira's embedding theorem.
Ohsawa and Sibony proved the following Kodaira type embedding theorem.

\begin{Thm*}[\cite{ohsawa-sibony2000}*{Theorem 3}, refined in \cite{ohsawa2012}]
Let $M$ be a compact $\Cont^\infty$ Levi-flat CR manifold equipped with a $\Cont^\infty$ CR line bundle $L$. 
Suppose $L$ is positive along leaves, i.e., there exists a $\Cont^\infty$ hermitian metric on $L$ such 
that the restriction of the curvature form to each leaf is everywhere positive definite. 
Then, for any $\kappa \in \N$, $L$ is $\Cont^\kappa$-ample, i.e., there exists $n_0 \in \N$ such that one can find 
CR sections $s_0, \cdots, s_N$ of $L^{\otimes n}$, of class $\Cont^\kappa$, 
for any $n \geq n_0$, such that the ratio $(s_0 : \cdots : s_N)$ embeds $M$ into $\CP^N$.
\end{Thm*}

We can make the regularity $\kappa \in \N$ arbitrarily large although we need to take $n_0$ sufficiently large. 
A natural question is whether we can improve the regularity to $\kappa = \infty$. 
The answer is no, in general, as the following case-study tells us.

\begin{MainTheorem}
Let $\Sigma$ be a compact Riemann surface, and $\mathcal{D}$ a holomorphic disc bundle over $\Sigma$. 
Denote its associated compact $\Cont^\infty$ Levi-flat CR manifold by $M = \bd \mathcal{D}$ 
in the associated flat ruled surface $\pi \colon X \to \Sigma$. 
Take a positive line bundle $L$ over $\Sigma$. 
Suppose that $\mathcal{D}$ has a unique non $\pm$holomorphic harmonic section $h$ with $\rank_\R dh = 2$ on an open dense set. 
Then, $\pi^*L|M$ is positive along leaves, but never $\Cont^\infty$ ample.
\end{MainTheorem}

We can easily see that the pull-back bundle $\pi^*L|M$ is positive along leaves. 
Thus, this theorem states non $\Cont^\infty$-ampleness of such CR line bundles.

The assumption is fulfilled for the following explicit example (Example \ref{ex_disc_bundles}): 
Let $\Sigma$ be a compact Riemann surface of genus $\geq 2$. Fix an identification of a universal covering $\widetilde{\Sigma} \simeq \D$ 
and regard $\pi_1(\Sigma) \simeq \Gamma < \Aut(\D)$ as a Fuchsian representation of $\Sigma$. 
Take a non-trivial quasiconformal deformation of $\Gamma$, say $\rho \colon \Gamma \to \Aut(\D)$. 
Set $\mathcal{D} \eqdef \widetilde{\Sigma} \times \D / (z, \zeta) \sim (\gamma z, \rho(\gamma)\zeta)$ for $\gamma \in \Gamma$.

One of other research directions of the analogue of Kodaira's embedding theorem is 
the problem concerning on projective embedding of compact laminations; 
we can find similar phenomenon in the work of Forn{\ae}ss and Wold \cite{fornaess-wold2011}*{Theorem 5.1} 
where they study compact $\Cont^1$ hyperbolic laminations. 
We also refer the reader to related works by Gromov \cite{gromov1999}*{pp.401--402}, 
Ghys \cite{ghys1997}*{\S 7}, Deroin \cite{deroin2008} and Mart{\'{\i}}nez Torres \cite{martinez-torres2011}. 

The organization of this paper is as follows. 
In \S 2, we introduce basic notions on Levi-flat CR manifolds. 
In \S 3, we recall and refine a classification result of holomorphic disc bundles 
with an emphasis on {\em Takeuchi 1-completeness} of certain holomorphic disc bundles. 
This notion is also known as $\log \delta$-pseudoconvexity \cite{brinkschulte2004}
or strong Oka property \cite{harrington-shaw2007}, 
and would be of interest from the viewpoint of confoliation. 
In \S 4, we state a variant of Bochner-Hartogs type extension theorem for CR sections. 
We give a self-contained proof for the reader's convenience.
In \S 5, we prove Main Theorem and pose some further questions.


\section{Levi-flat CR manifolds}
Let us recall basic notions briefly. For simplicity, we discuss under the assumption that 
manifolds and bundles have at least $\Cont^\infty$-smoothness. 

\subsection{In terms of foliation}
A {\em $\Cont^\infty$ Levi-flat CR manifold} is a triple $(M, \mathcal{F}, J)$ where 
$M$ is a $\Cont^\infty$ manifold, 
$\mathcal{F}$ is a $\Cont^\infty$ foliation on $M$ of real codimension one (the {\em Levi foliation} of $M$), 
and $J$ is a $\Cont^\infty$ section of $\End(T\mathcal{F})$ that induces a complex structure on each leaf, 
i.e., $J^2 = -\id$ and smooth sections of $T^{1,0} \eqdef \Ker (J - i\id) \subset \C \otimes T\mathcal{F} \subset \C \otimes TM$
are closed under the Lie bracket. 

The simplest example, which provides the local structure of Levi-flat CR manifolds, 
is $M = \C^{n-1} \times \R$ where 
$\mathcal{F}$ is given by its leaves $\{ \C^{n-1} \times \{t\} \}_{t \in \R}$ and 
$J$ is induced from the standard complex structure of $\C^{n-1}$. 
Any $\Cont^\infty$ Levi-flat CR manifold can be constructed by gluing some open sets of $\C^{n-1} \times \R$ together 
using leafwise holomorphic $\Cont^\infty$ maps.

We say a function $f \colon M \to \C$ is a {\em CR function} if it is leafwise holomorphic.

\subsection{In terms of CR geometry}
We will investigate Levi-flat CR manifolds in complex manifolds. 
The terminology of CR geometry is suitable for this purpose.

A {\em CR manifold} ({\em of hypersurface type}) is a pair $(M, T^{1,0})$ where $M$ is a $\Cont^\infty$ manifold 
of dimension $2n - 1$, and $T^{1,0}$ is a subbundle of $\C \otimes TM$ of $\rank_\C$ $n-1$
that satisfies $T^{1,0} \cap \overline{T^{1,0}} = 0$ and smooth sections of $T^{1,0}$ 
are closed under the Lie bracket. 
It models a real hypersurface $M$ in an $n$-dimensional complex manifold $(X, J_X)$; 
for such $M$ we can put $T^{1,0} \eqdef T^{1,0} X \cap \C TM \simeq $(the maximal $J_X$-invariant subspace of $TM$). 
Moreover, if the real hypersurface $M$ is given by a $\Cont^\infty$ defining function $r$, i.e., 
$r \colon M \subset U \to \R$ with $ M = \{ z \in U \mid r(z) = 0 \}$ and $dr \neq 0$ on $M$, 
we have $T^{1,0} = \Ker \d' r \subset T^{1,0} X$. 

We can redefine a $\Cont^\infty$ {\em Levi-flat CR manifold} as a $\Cont^\infty$ CR manifold $(M, T^{1,0})$ such that 
smooth sections of $T^{1,0} + \overline{T^{1,0}}$ are closed under the Lie bracket. 
The Levi foliation $\mathcal{F}$ is recovered by integrating the distribution $(T^{1,0} + \overline{T^{1,0}}) \cap TM$. 
In the case that $M$ is located in a complex manifold $X$ with defining function $r$, 
$M$ is Levi-flat if and only if the Levi form $i\d'\dbar r|T^{1,0} = 0$ as a quadratic form.  
This is the classical definition of {\em Levi-flat real hypersurface}.

We say a function $f \colon M \to \C$ is a {\em CR function} if it is annihilated by vectors of $T^{0,1} \eqdef \overline{T^{1,0}}$. 
If $M = \{ r = 0 \} \subset X$ and $f$ is of $\Cont^1$, 
it is equivalent to say that $\dbar \tilde{f}$ is proportional to $\dbar r$ on $M$ 
where $\tilde{f}$ is any $\Cont^1$ extension of $f$ on a neighborhood of $M$. 
In particular, the restriction of any holomorphic function defined near $M$ is CR. 
In the Levi-flat case, this definition agrees with the sense defined above.

\begin{Rem}
Only a few examples are known for {\em compact} Levi-flat real hypersurfaces. 
For instance, we cannot find such compact hypersurface in $\C^n$ 
since there is a strictly plurisubharmonic function $\sum_{i=1,\cdots,n} |z_i|^2$, 
whose restriction on each Levi leaf gives a strictly subharmonic function, and 
the maximum principle forbids the situation. 
The same reasoning implies no compact Levi-flat real hypersurface 
exists in Stein manifolds.
A famous still open conjecture is the non-existence of compact $\Cont^\infty$ Levi-flat 
real hypersurfaces in $\CP^2$. 
\end{Rem}

\subsection{CR line bundles}
Let $L$ be a $\Cont^\infty$ {\em CR line bundle} over a Levi-flat CR manifold $M$, that is 
a $\Cont^\infty$ $\C$-vector bundle of $\rank_\C$ 1 that possesses a trivialization cover 
whose transition functions are CR. 
Let $h$ be a $\Cont^\infty$ hermitian metric on $L$. 
We can find a connection on $L$ that equals to the Chern connection on $(L|N, h)$ along any leaf $N$.
We denote by $\Theta_h$ the curvature 2-form of the connection restricted along $T\mathcal{F}$, 
that is, $\Theta_h = - \d'_z \dbar_z \log h(z, t)$ where $(z, t) \colon M \supset U \to \C^{n-1} \times \R$
 is any foliated chart.
We say $L$ to be {\em positive along leaves} if there exists a hermitian metric $h$ on $L$ such that 
$i\Theta_h (\zeta, \overline{\zeta}) > 0$ for any non-zero $\zeta \in T^{1,0}$. 
If $M$ is three-dimensional, the existence of a positive-along-leaves CR line bundle over $M$
is equivalent to the tautness of its Levi foliation. (cf. \cite{martinez-torres2011}*{Lemma 1})


\section{Holomorphic disc bundles in flat ruled surfaces}
We recall a classification result on holomorphic disc bundles, 
with which a standard example of Levi-flat CR 3-manifolds associate, 
and supplement preceding results about pseudoconvexity of these spaces.

\subsection{Holomorphic disc bundles}
We begin by recalling a construction of holomorphic disc bundles.  
Let $\Sigma$ be a compact Riemann surface. 
A holomorphic fiber bundle over $\Sigma$ with fiber $\D \eqdef \{ \zeta \in \C \mid |\zeta| < 1 \}$ is 
called a {\em holomorphic disc bundle} over $\Sigma$. 
It can be easily seen that holomorphic trivializations form a flat trivializing cover, 
i.e., all of the transition functions are locally constant. 

Hence, any holomorphic disc bundle $\mathcal{D}$ can be obtained by the {\em suspension} construction: 
we can find a group homomorphism $\rho \colon \pi_1(\Sigma) \to \Aut(\D)$, called a {\em holonomy homomorphism}, 
giving a bundle isomorphism 
\begin{align*}
\mathcal{D} &\simeq \Sigma \times_\rho \D  \\
            &\eqdef \widetilde{\Sigma} \times \D / (z, \zeta) \sim (\gamma z, \rho(\gamma)\zeta) \text{ for $\gamma \in \pi_1(\Sigma)$}
\end{align*}
where $\widetilde{\Sigma}$ is a universal covering of $\Sigma$. We denote this disc bundle by $\mathcal{D}_\rho$.

The group $\Aut(\D)$ of biholomorphisms of $\D$ consists of M\"obius transformations 
preserving $\D$, acting on the Riemann sphere $\CP^1$ and fixing the unit circle $\bd \D$. 
Thus, it follows that a holomorphic disc bundle is canonically embedded in 
its associated flat ruled surface, say $\pi \colon X_\rho \eqdef \Sigma \timesrho \CP^1 \to \Sigma$, 
and the boundary of $\mathcal{D}_\rho$ in $X_\rho$, a flat circle bundle, 
becomes a compact $\Cont^\omega$ Levi-flat CR 3-manifold, 
say $M_\rho \eqdef \Sigma \timesrho \bd \D$. 
Note that $\mathcal{D}_\rho \to X \setminus \overline{\mathcal{D}_\rho}$, $(z, \zeta) \mapsto (z, 1/\overline{\zeta})$ 
is an anti-biholomorphism, which we call the {\em conjugation}.

\subsection{Classification}
Now we state a classification result of holomorphic disc bundles by means of harmonic sections, 
which is based on the works of Diederich and Ohsawa \cite{diederich-ohsawa1985}, \cite{diederich-ohsawa2007}.
\begin{Thm}
\label{diederich-ohsawa}
Let $\mathcal{D}$ be a holomorphic disc bundle over a compact Riemann surface $\Sigma$
and $M$ its associated Levi-flat CR 3-manifold. 
Then, one of the following cases occurs:
\begin{enumerate}
\item $\mathcal{D}$ admits a unique non-holomorphic harmonic section $h$ with $\rank_\R dh = 2$ on an open dense set.
\item $\mathcal{D}$ admits a unique locally non-constant holomorphic section.
\item $\mathcal{D}$ admits a unique harmonic section $h$ with $\rank_\R dh = 1$ on an open dense set.
\item $M$ admits one or two locally constant section(s).
\item $\mathcal{D}$ admits a locally constant section.
\end{enumerate}
\end{Thm}
Here a section is said to be {\em harmonic} if it is lifted to a $\rho$-equivariant harmonic map 
$\tilde{h} \colon\widetilde{\Sigma} \to \D$ where $\D$ is equipped with the Poincar\'e metric, and 
a section is said to be {\em locally constant} if it is locally constant in the (flat) trivializing coordinates. 

\begin{proof}[Proof of Theorem \ref{diederich-ohsawa}]
By \cite{diederich-ohsawa1985}*{Theorem 2}, we know that there exists either a harmonic section $h$ of $\mathcal{D}$, 
or a locally constant section of $M$. In the latter case, by examining possible holonomy homomorphisms, we know that 
there are at most two locally constant sections, which is the case (iv). (cf. \cite{diederich-ohsawa2007}*{Proposition 1.1})

Now we suppose the former case, the existence of such $h$. 
By applying a theorem of Sampson \cite{Sampson1978}*{Theorem 3} to the lift of $h$, 
we know that $\rank_\R\, dh$ is constant on an open dense subset of $\Sigma$. 
If the rank is zero or one, it is of the case (v) or (iii) respectively. 
The remaining case, where the rank is two, is classified into (ii) or (i) based on whether $h$ is holomorphic or not. 
\end{proof}

\begin{Ex}
\label{ex_disc_bundles}
We describe examples of each case in terms of holonomy homomorphism. 
\begin{enumerate}
\item Let $\Sigma$ be of genus $\geq 2$. Fix an identification $\widetilde{\Sigma} \simeq \D$ 
and regard $\pi_1(\Sigma) \simeq \Gamma < \Aut(\D)$ as a Fuchsian representation of $\Sigma$. 
Take a non-trivial quasiconformal deformation of $\Gamma$, say $\rho \colon \Gamma \to \Aut(\D)$. 
Then $\mathcal{D}_\rho$ is of the case (i). 
The unique harmonic section corresponds to the graph of the unique harmonic diffeomorphism $\Sigma = \D/\Gamma \to \D/\rho(\Gamma)$. 
\item Let $\Sigma$ and $\Gamma$ as above, and 
$\rho = \id \colon \Gamma \to \Gamma \subset \Aut(\D)$.
Then, $\mathcal{D}_\rho$ is of the case (ii). 
Its associated holomorphic section is obtained by the quotient of the diagonal set $\Delta \subset \widetilde{\Sigma} \times \D = \D \times \D$.
\item Let $\rho$ be a homomorphism from $\pi_1(\Sigma)$ to a subgroup of $\Aut(\D)$ generated by 
a hyperbolic element $H$ and an elliptic element of order two which reverses the axis of $H$. 
Then, $\mathcal{D}_\rho$ is of the case (iii). The image of the unique harmonic section corresponds to the axis of $H$.
\item Let $\rho$ be a homomorphism from $\pi_1(\Sigma)$ to an abelian subgroup of $\Aut(\D)$ 
that consists of parabolic (resp. hyperbolic) elements with common fixed point(s) on $\bd \D$. 
Then, $\mathcal{D}_\rho$ is of the case (iv). The locally constant section(s) correspond(s) to 
the suspension of the fixed point(s). 
\item Let $\rho$ be a homomorphism from $\pi_1(\Sigma)$ to an abelian subgroup of $\Aut(\D)$ 
that consists of elliptic elements with a common fixed point in $\D$, 
which is just isomorphic to the group of rotations ${\mathrm U}(1)$. 
Then, $\mathcal{D}_\rho$ is of the case (v). The suspension of the fixed point gives a locally constant section.
\end{enumerate}
Note that for the cases (iii)--(v), the examples above exhaust the cases, respectively.
\end{Ex}

\subsection{Pseudoconvexity} 
Dynamical complexity of the Levi foliations of the flat circle bundles is directly encoded in their holonomy homomorphisms. 
On the other hand, it is indirectly reflected on pseudoconvexity of the holomorphic disc bundles bounded by the flat circle bundles.

Known facts on the pseudoconvexity of a holomorphic disc bundle, say $\mathcal{D}$, over a compact Riemann surface 
is summarized as follows:
\begin{itemize}
\item In all the cases, $\mathcal{D}$ is weakly 1-complete (\cite{diederich-ohsawa1985}*{Theorem 1}).
\item In the cases (i)--(iv), $\mathcal{D}$ is 1-convex; it is particularly Stein 
if and only if it is of the cases (i), (iii) and (iv) (\cite{barrett1992}*{Theorem 2}).
\item In the cases (i) and (ii), $\mathcal{D}$ is Takeuchi 1-convex (\cite{diederich-ohsawa2007}*{Proposition 1.6}\footnote{
Its proof seems to contain some errors. 
}).
\end{itemize}

We will give a supplemental result for the case (i) using the following notion.

\begin{Def}[Takeuchi $q$-complete space]
Let $X$ be a complex manifold of dimension $n$ and $D$ a relatively compact domain in $X$ 
with $\Cont^2$ boundary. 
$D$ is said to be {\em Takeuchi $q$-complete} if there exists a $\Cont^2$ defining function $r$ of 
$\bd D$ defined on a neighborhood of $D$ with $D = \{ z \mid r(z) < 0 \}$ such that, 
with respect to a hermitian metric on $X$, 
at least $n - q + 1$ eigenvalues of the Levi form of $-\log(-r)$ are greater than 1 entire on $D$.
\end{Def}

This notion originates in the work of Takeuchi \cite{takeuchi1964} 
where he showed any proper locally pseudoconvex domain in $\CP^n$ acquires this property for $q = 1$. 
Although it has already had other names, $\log \delta$-pseudoconvexity in \cite{brinkschulte2004}, 
and the strong Oka condition in \cite{harrington-shaw2007}, 
we name it again in consideration of consistency with the terminologies 
such as $q$-convexity and $q$-completeness, and Takeuchi $q$-convexity in \cite{diederich-ohsawa2007}.

Takeuchi 1-completeness not only implies that the domain is Stein, 
but also implies that it behaves as if it is in complex Euclidean space: 
\begin{Thm}[\cite{ohsawa-sibony1998}*{Theorem 1.1}]
\label{completeness}
Let $D$ be a Takeuchi 1-complete domain with defining function $r$. 
Then, $-i\d' \dbar \log (-r)$ gives a complete K\"ahler metric on $D$, 
and it follows that $-(-r)^{t_0}$ with sufficiently small $t_0 > 0$ becomes 
a strictly plurisubharmonic bounded exhaustion function on $D$, 
i.e., $D$ is hyperconvex.
\end{Thm}

\begin{Rem}
From the viewpoint of confoliation \cite{eliashberg-thurston1998}*{Corollary 1.1.10}, 
we can translate a question on 
various strong pseudoconvexity of the complement of a Levi-flat real hypersurface 
into one 
on approximation of a foliation by contact structures. 
For example, suppose a compact Levi-flat real hypersurface $M$ has a Takeuchi 1-convex complement with defining function $r$. 
For small positive $\varepsilon$, the level sets $\{r = -\varepsilon \}$ are diffeomorphic to $M$ and 
possess contact structures induced from the strictly pseudoconvex CR structures. 
Thus, the family of the level sets defines a ``uniform'' contact deformation of the Levi foliation. 
Here ``uniform'' means that convergence to the foliation is {\em exactly} the same order entire on $M$. 
\end{Rem}

\subsection{Takeuchi 1-complete case}
\begin{Prop}
\label{takeuchi}
Let $\mathcal{D}$ be a holomorphic disc bundle over a compact Riemann surface $\Sigma$ 
with a uniquely determined non-holomorphic harmonic section $h$ with $\rank_\R dh = 2$ on an open dense set. 
Then, $\mathcal{D}$ is Takeuchi 1-complete in its associated ruled surface $X$.
\end{Prop}

\begin{proof}
Fix a finite open covering $\{U_\nu\}$ of $\Sigma$ giving trivializations of $\mathcal{D}$.
Set $\delta = \max_\nu \sup_{U_\nu} |h| < 1$ where the value of $h$ is taken with respect to the trivializing coordinate over each $U_\nu$.
It suffices to find a defining function $r$ of $\bd \mathcal{D}$ so that 
the eigenvalues of the complex Hessian of $-\log (-r)$ in each trivializing coordinate $(z, \zeta) \colon \pi^{-1}(U_\nu) \to \C^2 $ 
are bounded from below by a positive constant, 
since we can easily find a hermitian metric on $X$ 
that is comparable to $i (dz d\overline{z} + d\zeta d\overline{\zeta})$
by usual ``partition of unity'' argument.

We will find the desired $r$ in the form $r = r_0 e^{-\psi}$ 
where $r_0$ is the defining function of $\bd \mathcal{D}$ used in \cite{diederich-ohsawa1985}, 
and $\psi \colon \Sigma \to \R$ will be determined later. 
Recall the original defining function 
\[
r_0(z, \zeta) \eqdef \left|\frac{\zeta - h(z)}{1 - \overline{h(z)}\zeta}\right|^2 - 1 
\]
where $(z, \zeta)$ is any trivializing coordinate. 
It is clearly well-defined since the term inside the modulus is 
just a M\"obius transformation that maps $h(z)$ to $0$ and 
remaining choices of the fiber coordinate are only up to rotations. 

Take one of the trivializations, say $(z, \zeta) \colon \pi^{-1}(U_\nu) \to \C^2$. The Levi form is 
\begin{align*}
& i \d'\dbar (- \log(-r)) \\
& = i \d'\dbar \left( \psi - \log (1 - |\zeta|^2) - \log (1 - |h|^2) + 2 \mathrm{Re} \log (1 - \overline{h}\zeta) \right)\\
&= \left(\psi_{z\overline{z}} + (1 - |\zeta|^2) (|h_z|^2 + |h_{\overline{z}}|^2) + \frac{|\zeta - h |^2}{1 -|h|^2} | h_z - e^{2i \theta(z, \zeta)} \overline{h_{\overline{z}}} |^2 \right) 
    \frac{ idz \wedge d\overline{z} }{ |1 - \overline{h}\zeta|^2 (1 - |h|^2) } \\
&\qquad - h_z \frac{ i dz \wedge d\overline{\zeta} }{(1 - h\overline{\zeta})^2} 
    - \overline{h_z} \frac{ i d\zeta \wedge d\overline{z} }{(1 - \overline{h}\zeta)^2} 
    + \frac{i d\zeta \wedge d\overline{\zeta}}{(1 - |\zeta|^2)^2} 
\end{align*}
where $\theta(z, \zeta) \eqdef \arg (\zeta - h)/(1 - \overline{h}\zeta)$ and all the values on $h$ and $\psi$ are taken at $z$. 
We can check it by direct computation in three steps:
\begin{enumerate}
\item Fix $z_0 \in U$ in the trivialization. Choose a temporal trivializing coordinate $(z, \zeta^\natural)$ with $h^\natural(z_0) = 0$. 
\item Compute the Levi form on the fiber $\mathcal{D}_{z_0}$ in $(z, \zeta^\natural)$ coordinate. 
Note that the harmonicity of $h$ yields $h_{z\overline{z}}(z_0) = 0$.
\item Pull back the form to $(z, \zeta)$ coordinate.
\end{enumerate}

Now we are going to estimate the eigenvalues of the complex Hessian. 
The trace and determinant of the complex Hessian of $-\log (-r)$ are estimated as 
\begin{flalign*}
& (\text{trace of the complex Hessian}) &\\
& =     \frac{1}{(1 - |\zeta|^2)^2}  + \frac{ \psi_{z\overline{z}} + (1 - |\zeta|^2) (1 - |h|^2) (|h_z|^2 + |h_{\overline{z}}|^2) + |\zeta - h |^2 | h_z - e^{2i \theta} \overline{h_{\overline{z}}} |^2 }{ |1 - \overline{h}\zeta|^2 (1 - |h|^2)^2 } &\\
& \leq  \frac{1}{(1 - |\zeta|^2)^2} + \frac{\psi_{z\overline{z}} + (1 - |\zeta|^2) (|h_z|^2 + |h_{\overline{z}}|^2) + |\zeta - h |^2 | h_z - e^{2i \theta} \overline{h_{\overline{z}}} |^2 }{(1 - \delta)^{4} } &\\
& \leq  \frac{1}{(1 - |\zeta|^2)^2} + \frac{\psi_{z\overline{z}} + (1 - |\zeta|^2 + 2|\zeta - h |^2 ) (|h_z|^2 + |h_{\overline{z}}|^2) }{(1 - \delta)^{4} } &\\
& \leq  \frac{1}{(1 - |\zeta|^2)^2} + \frac{\psi_{z\overline{z}} + 8 (|h_z|^2 + |h_{\overline{z}}|^2) }{(1 - \delta)^{4} } &\\
& \leq  \frac{1}{(1 - |\zeta|^2)^2} + \sup_U \frac{\psi_{z\overline{z}} + 8 (|h_z|^2 + |h_{\overline{z}}|^2) }{(1 - \delta)^{4} } =\mathrel{\mathop{:}} \frac{1}{(1 - |\zeta|^2)^2} + C. &\\
\end{flalign*}
\begin{flalign*}
& (\text{determinant of the complex Hessian}) & \\ 
& = \frac{\psi_{z\overline{z}}}{(1 - |\zeta|^2)^2} +  \frac{1}{(1 - |\zeta|^2)^2} \left(  \frac{|\zeta - h|^2 | h_z - e^{2i\theta}\overline{h_{\overline{z}}}|^2}{|1 - \overline{h}\zeta|^2 (1 - |h|^2)^2} \right) & \\
& \qquad + \frac{1}{1 - |\zeta|^2} \left( \frac{|h_{\overline{z}}|^2}{| 1 - \overline{h}\zeta|^2 (1 - |h|^2) } + \frac{|\zeta - h|^2 |h_z|^2 }{|1 - \overline{h}\zeta|^4 (1 - |h|^2) } \right) & \\
   &\geq \frac{\psi_{z\overline{z}}}{(1 - |\zeta|^2)^2} +  \frac{|\zeta - h|^2 | h_z - e^{2i\theta}\overline{h_{\overline{z}}}|^2 }{4(1 - |\zeta|^2)^2} + \frac{|h_{\overline{z}}|^2}{4(1 - |\zeta|^2)} + \frac{|\zeta - h|^2|h_z|^2 }{16(1 - |\zeta|^2)} & \\
   &\geq \frac{\psi_{z\overline{z}}}{(1 - |\zeta|^2)^2} +  \frac{ |\zeta - h|^2 (| h_z | - |h_{\overline{z}}|)^2 + (1 - |\zeta|^2) |h_{\overline{z}}|^2 }{4(1 - |\zeta|^2)^2}& \\
   &\geq \frac{\psi_{z\overline{z}}}{(1 - |\zeta|^2)^2} +  \frac{ (|\zeta - h|^2 + 1 - |\zeta|^2) \min\{(| h_z | - |h_{\overline{z}}|)^2, |h_{\overline{z}}|^2 \}}{4(1 - |\zeta|^2)^2} &\\
   &\geq \frac{\psi_{z\overline{z}}}{(1 - |\zeta|^2)^2} +  \frac{ (1 - \delta)^2 \min\{(| h_z | - |h_{\overline{z}}|)^2, |h_{\overline{z}}|^2 \}}{4(1 - |\zeta|^2)^2}. &\\
\end{flalign*}
We will set $\psi$ to have sufficiently small range so that the trace is positive, 
in which situation the smaller eigenvalue $\lambda$ of the complex Hessian of $-\log (-r)$ is estimated as 
\begin{align*}
\lambda 
&=   \frac{\text{trace}}{2} - \sqrt{ \frac{\text{trace}}{2} - \det } \geq  \frac{\det}{\text{trace}} \\
& \geq  \frac{1}{1 + C}  \left(\psi_{z\overline{z}} + \frac{(1 - \delta)^2}{4} \min \{ |h_{\overline{z}}|^2, (|h_z| - |h_{\overline{z}}|)^2 \} \right).
\end{align*}
Note that this estimate does not depend on $\zeta$, and 
is sharp in the sense that 
the smaller eigenvalue of the complex Hessian of $-\log (-r_0)$, which corresponds to the second term in the estimate, 
actually equals to 0 at $(z, 0)$ if $h_{\overline{z}} (z) = 0$ and 
tends to 0 near some points of $\bd \mathcal{D}_{z}$ if $|h_z(z)| = |h_{\overline{z}}(z)|$, 
which facts can be deduced from the explicit formula of the Levi form. 
This situation leads us to modify $r_0$ with $\psi$ strictly subharmonic on such locus in $\Sigma$.

From Lemma \ref{isolated} below and the assumption on rank of $dh$, we can find a non-empty relatively compact set $V \subset \Sigma$ on which both $|h_{\overline{z}}|$ and $|h_z| - |h_{\overline{z}}|$ never vanish. 
Removing a relatively compact $W \subset V$ from $\Sigma$, we obtain an open Riemann surface $\Sigma \setminus \overline{W}$, 
which carries a strictly subharmonic exhaustion function $\psi_0$. 
We extend $\psi_0 | \Sigma \setminus V$ to $\Sigma$ so as to vanish on $W$, say $\psi_1$. 
We take $0 < c \ll 1$ for $\psi \eqdef c \psi_1$ to satisfy, in the all of the trivializing coordinates, 
$\psi_{z\overline{z}}(1 - \delta)^{-4} > -1$, and 
\[
\psi_{z\overline{z}} + \frac{(1 - \delta)^2}{4} \min \{ |h_{\overline{z}}|^2, (|h_z| - |h_{\overline{z}}|)^2\} > 0 \text{\quad on $V$.}
\]
Using this $\psi$, we have obtained the desired defining function $r$. 
\end{proof}

\begin{Lem}
\label{isolated}
Let $\mathcal{D}, \Sigma,$ and $h$ as in Proposition \ref{takeuchi}. Then, the zero set of $h_{\overline{z}}$ is finite.
\end{Lem}

\begin{proof}
We have well-defined forms $|h_z|(1 - |h|^2)^{-1} |dz|$, $|h_{\overline{z}}|(1 - |h|^2)^{-1} |dz|$ 
and $\mathrm{Hopf}(h) \eqdef h_z \overline{h_{\overline{z}}} (1 - |h|^2)^{-2} dz^2$ on $\Sigma$. 
The harmonicity of $h$ is equivalent to holomorphicity of $\mathrm{Hopf}(h)$, whose zero set consists of 
$4(\text{genus of }\Sigma) - 4$ points. (Note that the assumption implies that $\pi_1(\Sigma)$ is non-abelian, thus the genus of $\Sigma > 1$.)
Therefore the zero set of $h_{\overline{z}}$ is also finite.

\end{proof}

\begin{Que}
What about the case (iii)? We know an example in which $\mathcal{D}_\rho$ is Stein but not Takeuchi 1-complete (\cite{ohsawa-sibony1998}*{Theorem 1.2}). 
\end{Que}

\section{Bochner-Hartogs type extension theorem}
We will state a Bochner-Hartogs type extension theorem for CR sections 
of finite regularity, which can be obtained by established procedures as 
in \cite{ohsawa1999}, \cite{brinkschulte2004} and \cite{chakrabartishaw2012}.
Here we give a simple proof for the reader's convenience. 

\begin{Thm}
\label{extension}
Let $X$ be a connected compact complex manifold of dimension $n \geq 2$, 
$L$ a holomorphic line bundle over $X$, and 
$M$ a $\Cont^\infty$ compact Levi-flat real hypersurface of $X$ 
which splits $X$ into two Takeuchi 1-complete domains $D \sqcup D'$. 
Then, there exists $\kappa \in \N$ such that 
any $\Cont^\kappa$ CR section of $L|M$ extends to a holomorphic section of $L$.
\end{Thm}

\begin{proof}
We set 
\[
N_0 \eqdef \min \left\{ N \in \N \, \middle|
\begin{array}{l}
  i\Theta_{h_0} - N i \d' \dbar (- \log (-r)) < 0  \text{\quad on $D$, } \\
  i\Theta_{h_0} - N i \d' \dbar (- \log (-r')) < 0 \text{\quad on $D'$, } \\
  \text{$h_0$: hermitian metric of $L$}, \\
  \text{$r$ (resp. $r'$): defining function of $M$} \\
  \text{\quad which makes $D$ (resp. $D'$) Takeuchi 1-complete}
\end{array}
\right\}.
\]
The assumption yields $N_0 < \infty$. Put $\kappa \eqdef \lceil n + 1 + N_0/2 \rceil (\geq 4)$. 
Take $h_0$, $r$, and $r'$ to attain the minimum, 
and fix an arbitrary hermitian metric $g_0$ of $X$.

We denote by $\langle \cdot, \cdot \rangle_{g_0, h_0}$ (resp. $| \,\cdot\, |_{g_0, h_0} $)
the fiber metric (resp. norm) of $L \otimes \bigwedge \C TX^*$ determined by $g_0$ and $h_0$, 
and write $d \text{vol}_{g_0}$ for the volume form on $X$ determined by $g_0$. 
Integration with respect to these metrics is denoted by 
\[
\llangle \omega, \eta \rrangle_{g_0, h_0, D} \eqdef \int_D \langle \omega, \eta \rangle_{g_0, h_0} d\text{vol}_{g_0}
\]
and write $\| \omega \|^2_{g_0, h_0, D} \eqdef \langle \omega, \omega \rangle_{g_0, h_0, D}$. 
We also use the following notation for function spaces.
\begin{itemize}
\item $\Cont^\kappa_{(p,q)}(X, L)$: the space of $L$-valued $\Cont^\kappa$ $(p, q)$-forms over $X$.
\item $\Cont^\kappa_{0, (p,q)}(D, L)$: the space of $L$-valued compactly supported $\Cont^\kappa$ $(p, q)$-forms over $D$.
\item $L^2_{(p,q)}(D, L; g_0, h_0)$: the space of $L$-valued measurable $(p, q)$-forms over $D$ whose $\| \cdot \|_{g_0, h_0, D}$ norm is finite. 
\end{itemize}
We will omit the subscript $(p,q)$ when $(p, q) = (0, 0)$.

The proof is separated into three lemmas. 

\begin{Lem}
\label{finite_order_extension}
Let $s$ be a $\Cont^\kappa$ CR section of $L|M$. 
Then we can extend $s$ to $\tilde{s} \in \Cont^{2}(X, L)$ so that 
\begin{equation}
\label{vanishing_order}
| \dbar \tilde{s} |_0 \eqdef | \dbar \tilde{s} |_{g_0, h_0} = O(r^{\kappa - 2}) \text{\quad along $M$}
\end{equation}
where $r$ is any $\Cont^\infty$ defining function of $M$. 
\end{Lem}

\begin{proof}[Proof of Lemma \ref{finite_order_extension}]
Firstly, we extend $s$ to a $\Cont^\kappa$ section of $L$, still denoted by $s$, 
using a $\Cont^\infty$ collaring $M \times (-\epsilon, \epsilon) \to X$ of $M$ 
and a transversal cut-off function with enough small support. 
Since $s|M$ is CR, we can find a $\Cont^{\kappa-1}$ section of $L|M$, say $\alpha_1$, 
such that $\dbar s = \alpha_1 \dbar r$ on $M$.  
We extend $\alpha_1$ to a $\Cont^{\kappa-1}$ section of $L$. Put $s_1 \eqdef s - \alpha_1 r$. 
Then, $| \dbar s_1 |_0 = |(\dbar s - \alpha_1 \dbar r) - \dbar \alpha_1 r |_0 = O(r)$ 
because $\dbar s_1$ vanishes on $M$ and is of class $\Cont^{\kappa-2}$. 

Suppose we have inductively constructed a $\Cont^{\kappa-\ell}$ extension $s_\ell$ of $s$ 
with $s_\ell =  s - \alpha_1 r - \alpha_2 r^2/2 - \cdots - \alpha_\ell r^{\ell}/\ell$ and $|\dbar s_\ell |_0 = O(r^\ell)$. 
Write $\dbar s_\ell  = \beta_\ell r^\ell $ with $\beta_\ell \in \Cont^{\kappa-(\ell+1)}_{(0,1)}(X, L)$. 
We obtain $0 = \dbar^2 s_\ell = \dbar \beta_{\ell} r + \dbar r \wedge \beta_\ell$. 
Thus, we can find $\alpha_{\ell+1} \in \Cont^{\kappa-(\ell+1)}(X, L)$ such that 
$\beta_\ell = \alpha_{\ell+1} \dbar r$ on $M$. 
Putting  $s_{\ell+1} \eqdef s_{\ell} - \alpha_{\ell+1} r^{\ell+1}/(\ell + 1)$ 
gives $| \dbar s_{\ell+1} |_0 = | (\beta_\ell  -  \alpha_{\ell+1} \dbar r) r^\ell - \dbar \alpha_{\ell+1} r^{\ell+1} |_0 =  O(r^{\ell+1})$ 
while $\beta_\ell - \alpha_{\ell + 1} r^\ell$ is differentiable, which holds if $\kappa - (\ell + 1) \geq 1$.

Letting $\tilde{s} \eqdef s_{\kappa - 2}$ completes the proof.
\end{proof}

We perform a correction to $\tilde{s}$ to obtain the desired holomorphic extension. 
Once we solve the $\dbar$-equation $\dbar u = \dbar \tilde{s}$ on $X$ in the distribution sense 
with the condition $u|M = 0$, we obtain the desired extension $\tilde{s} - u $ 
since holomorphic functions are characterized as weak solutions of the Cauchy-Riemann equation. 

By Theorem \ref{completeness}, $i\d'\dbar (- \log (-r))$ defines a complete K\"ahler metric $g$ on $D$, 
which blows up in $O(r^{-2})$ along $M$. 
Consider the hermitian metric $h = h_0 r^{-N_0}$ on $L$. 
The condition (\ref{vanishing_order}) on $\tilde{s}$ implies 
\begin{align*}
\| \dbar \tilde{s} \|^2_{g, h} 
&\eqdef \| \dbar \tilde{s} \|^2_{g, h, D} \\
&= \int_D | \dbar \tilde{s} |^2_{g, h} d\text{vol}_g
= \int_D O(r^{2(\kappa - 2)}) O(r^2) O(r^{-N_0}) O(r^{-2n})
< \infty, 
\end{align*}
i.e., $\dbar \tilde{s} \in L^2_{(0,1)}(D, L; g, h)$. 
We can solve $\dbar u = \dbar \tilde{s}$ on $D$ 
thanks to the following $L^2$ cohomology vanishing theorem.

\begin{Lem}
\label{vanishing}
For any $v \in L^2_{(0,1)}(D, L; g, h)$ with $\dbar v = 0$, 
there exists a solution $u \in L^2(D, L; g, h)$ of $\dbar u = v$ 
in the sense that there exists a sequence 
$u_n \in \Cont^\infty_0 (D, L)$ such that 
$u_n \to u$ in $L^2(D, L; g, h)$ and $\dbar u_n \to v$ in $L^2_{(0,1)}(D, L; g, h)$.
\end{Lem}

\begin{proof}[Proof of Lemma \ref{vanishing}]
By the standard $L^2$ method of Andreotti-Vesentini, 
the conclusion follows from the following estimate 
\[
\| \dbar u \|^2_{g, h} + \| \dbar^*_{g, h} u \|^2_{g, h} \gtrsim \| u \|^2_{g, h}
\]
for $u \in C_{0, (0, 1)}^\infty(D, L)$. Here we denote by $\dbar^*_{g, h}$ the formal adjoint 
of the operator $\dbar \colon L^2(D, L; g,h) \to L^2_{(0,1)}(D, L; g, h)$.
Note that we have used the completeness of $g$ to obtain the solution 
not only in the sense of distribution but also in the sense above. 

By the Nakano inequality, we achieve the estimate as follows:
\begin{align*}
\| \dbar u \|^2_{g, h} + \| \dbar^*_{g, h} u \|^2_{g, h}
& \gtrsim {\llangle} [i\Theta_h, \Lambda]u, u \rrangle_{g, h} = - \llangle i\Theta_h u, L u \rrangle_{g, h} \\
& \gtrsim -\min \left\{
\begin{array}{l}
\text{sum of the $(n-1)$ eigenvalues of $i\Theta_h$} \\
\text{with respect to $g$}
\end{array}
\right\}  \| u \|^2_{g, h}. 
\end{align*}
The eigenvalues of $i\Theta_h$ with respect to $g$ tend to $-N_0$ near $M$. 
It follows that the RHS $\gtrsim \| u \|^2_{g, h}$.
\end{proof}

Performing the same procedure on $D'$, we obtain a section 
$u$ of $L|D \sqcup D'$ with $\dbar u = \dbar \tilde{s}$ on $D \sqcup D'$ in the sense above. 
Consider the zero extension of $u$ on $X$, still denoted by $u$. 
The following lemma completes the proof of Theorem \ref{extension}.

\begin{Lem}
$\dbar u = \dbar \tilde{s}$ on $X$ in the sense of distribution. 
\end{Lem}
\begin{proof}
Let $u_n \in \Cont^\infty_0(D \sqcup D', L) \subset \Cont^\infty(X, L) $ be the approximation of $u$ found in Lemma \ref{vanishing}. 
Since $L^2(D, L; g_0, h_0) \hookrightarrow L^2(D, L; g, h)$ is continuous, we have
$u_n \to u$ in $L^2(D \sqcup D', L; g_0, h_0) \simeq L^2(X, L; g_0, h_0)$. 

Take a test function $\phi \in \Cont^\infty(X, L)$. Denote by $\dbar^*_0$ the formal adjoint 
of the operator $\dbar \colon L^2(X, L; g_0,h_0) \to L^2_{(0,1)}(X, L; g_0, h_0)$. Then, 
\begin{align*}
\llangle \dbar u - \dbar \tilde{s}, \phi \rrangle_{g_0, h_0, X}
& = \llangle u, \dbar^*_0 \phi \rrangle_{g_0, h_0, X} - \llangle \dbar \tilde{s}, \phi \rrangle_{g_0, h_0, X}\\
& = \lim_{n \to \infty}\llangle u_n, \dbar^*_0 \phi \rrangle_{g_0, h_0, X} - \llangle \dbar \tilde{s}, \phi \rrangle_{g_0, h_0, X}\\
& = \lim_{n \to \infty}\llangle u_n, \dbar^*_0 \phi \rrangle_{g_0, h_0,  D \sqcup D'} - \llangle \dbar \tilde{s}, \phi \rrangle_{g_0, h_0, D \sqcup D'}\\
& = \lim_{n \to \infty}\llangle \dbar u_n - \dbar \tilde{s}, \phi \rrangle_{g_0, h_0,  D \sqcup D'} \\
& = 0.
\end{align*}
It completes the proof.
\end{proof}

\end{proof}

\begin{Cor}
Suppose $X$, $L$, $M$, and $\kappa$ as in Theorem \ref{extension}. 
Then, all of the $\Cont^{\kappa}$ CR sections of $L|M$ are automatically of class $\Cont^{\infty}$,
and they form a finite dimensional vector space.
\end{Cor}

We will use the following form of Theorem \ref{extension} in the proof of Main Theorem.
\begin{Cor}
\label{smooth_extension}
Suppose $X$, $L$ and $M$ as in Theorem \ref{extension}. 
Then, any $\Cont^\infty$ CR section of $L|M$ extends to a holomorphic section of $L$.
\end{Cor}

\section{Proof of Main Theorem}

\begin{proof}[Proof of Main Theorem]
From Proposition \ref{takeuchi}, $\mathcal{D}$ is Takeuchi 1-complete. 
The harmonic section of $X \setminus \overline{\mathcal{D}}$ is obtained 
by conjugating the harmonic section of $\mathcal{D}$. 
Thus, $X \setminus \overline{\mathcal{D}}$ is also Takeuchi 1-complete.
Hence, Corollary \ref{smooth_extension} implies that for any $n \geq 1$,
all of the $\Cont^\infty$ CR sections of $(\pi^*L|M)^{\otimes n}$ extend to 
holomorphic sections of $(\pi^*L)^{\otimes n}$.

On the other hand, $\pi^* \colon H^0(\Sigma, L^{\otimes n}) \to H^0(X, (\pi^*L)^{\otimes n})$ 
gives an isomorphism. Since we can give a trivializing cover of $(\pi^*L)^{\otimes n}$ by 
pulling back that of $L$, 
and the sections should be constant along any fiber $\pi^{-1}(p) \simeq \CP^1$ in these trivializations.
Hence it is impossible for the sections in $H^0(X, (\pi^*L)^{\otimes n})$ 
to separate points in the same fiber for any $n$. 
Therefore, we cannot make a projective embedding by any ratio of those sections. 
\end{proof}

We conclude this paper with further questions.

\begin{Que}
Can we prove Main Theorem intrinsically, i.e., without looking the natural Stein filling? 
\end{Que}

\begin{Que}
Let $M$ be a compact Levi-flat CR manifold, and $L$ a CR line bundle over $M$. 
We define the threshold regularity $\kappa(M, L)$ to be the minimal $\kappa \in \N \cup \{ \infty \}$, if exists,
so that $\Cont^\kappa$ CR sections of $L$ form a finite dimensional vector space.
In the situation illustrated in Main Theorem,  
the proof of Theorem \ref{extension} indicates that 
$\kappa(M, (\pi^*L|M)^{\otimes n})$ is well-defined and  
$\kappa(M, (\pi^*L|M)^{\otimes n}) = O(n)$ as $n \to \infty$. 
On the other hand, Ohsawa-Sibony's projective embedding theorem implies that 
 $\kappa(M, (\pi^*L|M)^{\otimes n}) \to \infty$ as $n \to \infty$. 
Can we read any dynamical property of the Levi foliation 
from the asymptotic behavior of the $\kappa(M, (\pi^*L|M)^{\otimes n})$?
\end{Que}


\subsection*{Acknowledgments}
The author would like to express his profound gratitude to his advisor T.\ Ohsawa. 
He is grateful to J.\ Brinkschulte for pointing out a mistake in a previous version of this article
and also grateful to B.-Y.\ Chen for useful comments on hyperconvexity.
He thanks R. Kobayashi and K. Matsumoto for helpful comments which 
improved the presentation of this paper. 
Part of this work was done during 
``Ecole d'\'et\'e 2012: Feuilletages, courbes pseudoholomorphes, applications'' at Institute Fourier, 
and ``Nagoya-Tongji joint workshop on Bergman kernel'' at Tongji University. 
He is grateful to both institutes for their hospitality and support.


\end{document}